\documentclass[12pt]{amsart}
\usepackage{amssymb,amsfonts,amsmath,amsopn,amstext,amscd,latexsym, amsthm, enumerate,mathrsfs}

\usepackage{color}
\usepackage{xcolor}
\usepackage{hyperref}
\hypersetup{
     colorlinks   = true,
     citecolor    = blue
}

\usepackage[margin=30mm]{geometry}
\usepackage{bbm}
\usepackage[all,cmtip]{xy}

\usepackage{enumerate}
\usepackage[shortlabels]{enumitem}

\newtheorem{theorem}{Theorem}[section]

\newtheorem{lemma}[theorem]{Lemma}

\newtheorem{proposition}[theorem]{Proposition}

\theoremstyle{remark}

\DeclareMathOperator{\Irr}{Irr}

\numberwithin{equation}{section}

\begin{document}

\title[Equivalent version of Huppert's conjecture for $K_3$-groups]{Equivalent version of Huppert's conjecture for $K_3$-groups}

\author[M.~Ghasemi]{Mohsen Ghasemi}
\address{Department of Mathematics, Urmia University,  Urmia  57135, Iran}
\email{m.ghasemi@urmia.ac.ir}

\author[S.~Hekmatara]{Somayeh Hekmatara}\address{Department of Mathematics, Urmia University,  Urmia  57135, Iran}
\email{somayehhekmatara93@gmail.com}

\thanks{}


\date{\today}

%\date{June 9, 2013}
%\thanks{}


\begin{abstract}
For a character $\chi$ of a finite group $G$ the number ${\rm cod(\chi)}=|G:{\rm
Ker}(\chi)|/\chi(1)$ is called the codegree of $\chi$. In this note, we verify  the equivalent version of Huppert's conjecture for $K_3$-groups.
\end{abstract}

\maketitle
\section{Introduction}
Throughout this paper, $G$ will be a finite group and ${\rm
cd}(G)$ will be the set of the degrees of the complex irreducible
characters of $G$.  The number ${\rm cod(\chi)}=|G:{\rm
ker}(\chi)|/\chi(1)$ is called the codegree of $\chi$. Also the set
of codegrees of the irreducible characters of $G$ is denoted by
${\rm cod}(G)$. This definition for codegrees first appeared in
\cite{QWW}, where the authors use a graph-theoretic approach to
compare the structure of a group with its set of codegrees.
In \cite{ABGGH, AGG, BA, SAKh} some properties of the codegrees of irreducible characters of finite groups have been studied.

In $1990$, Huppert proposed the following conjecture:\\
{\bf Huppert's Conjecture.} Let $H$ be any finite non-abelian simple group and $G$ a finite
group such that ${\rm cd}(G) = {\rm cd}(H)$. Then, $G\cong H \times A$, where $A$ is abelian.

Many people were devoted to the study of this problem. An analogues of Huppert’s
conjecture can be proposed and studied for any set of integers related to a finite group.
For instance, a dual version of Huppert’s conjecture for the set of conjugacy class sizes is
considered. For more results see  \cite{A, AKMT, AK}.  In this paper, we are concerned with the following conjecture, inspired by Huppert's
conjecture:\\
{\bf Conjucture:} Let $G$ be a finite group and $H$ a non-abelian simple group. If ${\rm cod}(G) =
{\rm cod}(H)$, then $G \cong H$. \\
This conjecture first proposed in \cite{BAK}, where the authors  verified the above conjecture for all projective special linear groups
of degree $2$.  In this article by using the same method in \cite{BAK}, we verify this conjecture for  $K_3$-groups. Also we recall that Huppert's conjecture is true for $K_3$-groups (see \cite{H}).

We refer to \cite{Isa} for the notation of character theory of finite groups.  If $L\leqslant G$ and $\theta\in\Irr(L)$, then we write $\Irr(G \mid \theta)=\{\chi\in\Irr(G) \mid [\chi_L, \theta]\neq 0\}$. Moreover, ${\rm Irr}(G \mid L)={\rm Irr}(G)-{\rm Irr}(G/L)$.

\section{Preliminaries}
In this section we give some results which will be used later in our proofs.

\begin{proposition}{\rm\cite[ Lemma 2]{MM}}\label{f1}
Let $S$ be a non-abelian finite simple group. Then there exists $1_s \neq \theta \in {\rm Irr}(S)$ that extends to ${\rm Aut}(S)$.

\end{proposition}

\begin{proposition}{\rm\cite[ Lemma 5]{BCLP}}\label{f2}
Let $N$ be a minimal normal subgroup of $G$ such that $N=S_1 \times S_2 \times \cdots \times S_t$, where $S_i \cong S$ is a non-abelian simple group. If $\theta \in {\rm Irr}(S)$ extends to ${\rm Aut(S)}$ then $\theta \times \theta \times \cdots \times \theta \in {\rm Irr}(N)$ extends to $G$.

\end{proposition}

\begin{proposition}{\rm\cite[ Theorem C]{MM}}\label{d1}
Let $G$ be a nonabelian finite simple group. Then either $|cd(G)| \geq 8$ or one of the following cases holds.
\begin{description}
\item[(i)]  $|{\rm cd}(G)|=4$ and $G={\rm PSL(2, 2^f)}$, where $f \geq 2$.
\item[(ii)]  $|{\rm cd}(G)|=5$ and $G={\rm PSL(2, q)}$, where $q=p^f$ , $p \neq 2$ and $p^f > 5$. 
\item[(iii)] $|{\rm cd}(G)|=6$  and  $G={\rm PSL}(3,4)$ or $G/N=\hskip 0.1 cm^2B_2(2^{2f+1})$, where $f \geq 1$ .
\item[(iv)]  $|{\rm cd}(G)|=7$  and $G={\rm PSL}(3,3)$, $A_7$, $M_{11}$ or $J_1$.
\end{description}

\end{proposition}

{\bf Remark.} By \cite{CCNPW} we see that ${\rm cod}(U_3(3))=\{ 1, 2^4.3^2.7, 2^5.3^3, 2^4.3^3, 2^5.3^2, 2^5.7, 2^3.3^3, 3^3.7\}$. Also ${\rm cod}(U_4(2))=\{1, 2^6.3^4, 2^5.3^3.5, 2^5.3^4, 2^6.3^3, 2^4.3^4,
 2^3.3^3.5, 2^5.3^3, 2^3.3^4, 2^6.3^2, 2^4.3^3, 3^4.5, 2^6.5\}.$

By ~\cite[Lemma 2.1]{QWW} we have the following result.\\
\begin{proposition}\ \  \label{NC1}
Let $\chi \in {\rm Irr}(G)$.
\begin{itemize}
\item[(i)] For any $N\unlhd G$ with $N\leq {\rm ker(\chi)}$, $\chi$ may be viewed as an irreducible character of $G/N$.
 The codegree ${\rm cod(\chi)}$ of $\chi$ is the same whenever $\chi$ is seen as an irreducible character of $G$ or
 $G/N$. Furthermore, ${\rm cod(\chi)}$ is independent of the choice
 of such $N$. 
\item[(ii)] If $M$ is a subnormal subgroup of $G$ and $\varphi$ is an irreducible constituent of $\chi_M$, then ${\rm cod(\varphi)}$ divides ${\rm cod(\chi)}$.
\end{itemize}
\end{proposition}

%%%%%%%%%%%%%%%%%%%%%%%%%%%%%%%%%%%%%%%%%%%%%%%%%%%%%%%
%%%%%%%%%%%%%%%%%%%%%%%%%%%%%%%%%%%%%%%%%%%%%%%%%%%%%%%
%%%%%%%%%%%%%%%%%%%%%%%%%%%%%%%%%%%%%%%%%%%%%%%%%%%%%%%
%%%%%%%%%%%%%%%%%%%%%%%%%%%%%%%%%%%%%%%%%%%%%%%%%%%%%%%
%\section{Preliminaries}

%%%%%%%%%%%%%%%%%%%%%%%%%%%%%%%%%%%%%%%%%%%%%%%%%%%%%%%
%%%%%%%%%%%%%%%%%%%%%%%%%%%%%%%%%%%%%%%%%%%%%%%%%%%%%%%
%%%%%%%%%%%%%%%%%%%%%%%%%%%%%%%%%%%%%%%%%%%%%%%%%%%%%%%
%%%%%%%%%%%%%%%%%%%%%%%%%%%%%%%%%%%%%%%%%%%%%%%%%%%%%%%
\section{main results}
We recall that if $G$ is a finite simple $K_3$-group, then $G$ is isomorphic to one of the following
groups:\\
$$A_5, A_6, {\rm L_2(7)}, {\rm L_2(8)}, {\rm L_2(17)}, {\rm L_3(3)},
{\rm U_3(3)}, {\rm U_4(2)}.$$
 Also by \cite{BAK}, we see that the conjecture is true for $A_5, A_6, {\rm L_2(7)}, {\rm L_2(8)}, {\rm L_2(17)}, {\rm L_3(3)}$. Thus in this section we just verify the conjecture for ${\rm U_3(3)}$ and ${\rm U_4(2)}$. 
\begin{theorem}\label{a1}
Suppose that $G$ is a group  and ${\rm cod}(G)={\rm cod}(U_3(3))$. If $N$ is a maximal normal subgroup of $G$ then $G/N \cong  U_3(3)$.

\end{theorem}

\begin{proof}
By our assumption $G/N$ is a simple group. If $G/N$ is abelian then $G^{'} \leq N$ and so $G \neq G^{'}$. Now by Proposition~\ref{NC1}, ${\rm cod}(G/G^{'}) \subseteq {\rm cod}(G)$ and so ${\rm cod}(G)$ contains  a prime power, a contradiction. Thus we may suppose that $G/N$ is nonabelian simple group. If $|{\rm cd}(G/N)| \geq 8$ then by ${\rm cod}(G/N) \subseteq {\rm cod}(G)$ and ${\rm cod}(G)={\rm cod}(U_3(3))$ we see that $|{\rm cod}(G/N)|=|{\rm cod}(U_3(3))|=8$. Thus  $|{\rm cd}(G/N)|=|{\rm cd}(U_3(3))|$ and by Huppert's conjecture $G/N \cong U_3(3)$.  Thus we may suppose that  $|{\rm cd}(G/N)| < 8$ and we find contradiction.  Now by Proposition~\ref{d1}, we considering the following cases.\\
{\bf Case 1.} $|{\rm cd}(G/N)|=4$ and $G/N \cong {\rm PSL(2, 2^f)}$, where $f \geq 2$.\\
By \cite{W}, we know that ${\rm cod}({\rm PSL(2, 2^f)})=\{1, q(q-1), q(q+1), q^2+1 \}$, where $q=2^f$ and ${\rm cod}(U_3(3))=\{1, 2^4.3^2.7, 2^5.3^3, 2^4.3^3, 2^5.3^2, 2^5.7, 2^3.3^3, 3^3.7 \}$. It is easy to see that ${\rm cod}({\rm PSL(2, 2^f)}) \nsubseteq {\rm cod}(U_3(3))$, a contradiction.\\
{\bf Case 2.} $|{\rm cd}(G/N)|=5$ and $G/N \cong {\rm PSL(2, q)}$, where $q=p^f$,  $p \neq 2$ and $q> 5$.\\
 By \cite{W}, we know that  ${\rm cod(PSL(2, q))} = \{1, q(q-1)/2, q(q+1)/2, (q^2-
1)/2,  q(q-\epsilon(q))\},$ where $\epsilon(q) = (-1)^{(q-1)/2}$.
It is easy to see that ${\rm cod}({\rm PSL(2, p^f)}) \nsubseteq {\rm cod}(U_3(3))$, a contradiction. \\
{\bf Case 3.} $|{\rm cd}(G/N)|=6$  and   $G/N={\rm PSL}(3,4)$ or $G/N=\hskip 0.1 cm ^2B_2(2^{2f+1})$, where $f \geq 1$.\\
First suppose that $G/N={\rm PSL}(3,4)$. We know that ${\rm cd}({\rm PSL}(3,4))=\{1, 20, 35, 45, 63, 64\}$ and so ${\rm cod}({\rm PSL}(3,4))=\{1, 2^4.3^2.7, 2^6.3^2, 2^6.7, 2^6.5, 3^2.5.7\}$. Now $2^6.5 \in {\rm cod}({\rm PSL}(3,4))$ and so $2^6.5 \notin {\rm cod}(U_3(3))$, a contradiction. Thus we may suppose  $G/N=\hskip 0.1 cm ^2B_2(2^{2f+1})$,  where $q^2=2^{2f+1}$, $r=2^{f+1}$ and $f \geq 1$.  We know that ${\rm cd}(^2B_2(2^{2f+1}))=\{1, q^4, q^4+1, (q^2-r+1)(q^2-1), (q^2+r+1)(q^2-1), r(q^2-1)/2\}$ and so ${\rm cod}(^2B_2(2^{2f+1}))=\{1, (q^4+1)(q^2-1), q^4(q^2-1), \frac{q^4(q^4+1)}{q^2-r+1}, \frac{q^4(q^4+1)}{q^2+r+1}, \frac{2q^4(q^4+1)}{r}\}$. Now with the simple check we can see that ${\rm cod}(^2B_2(2^{2f+1}))\nsubseteq {\rm cod}(U_3(3))$, a contradiction.\\
{\bf Case 4.} $|{\rm cd}(G/N)|=7$  and $G/N={\rm PSL}(3,3)$, $A_7$, $M_{11}$ or $J_1$.\\
We know that ${\rm cd}({\rm PSL}(3,3))=\{1, 12, 13, 16, 26, 27, 39\}$ and so ${\rm cod}({\rm PSL}(3,3))=\{1, 2^2.3^2.13,\\ 2^4.3^3, 3^3.13,
 2^3.3^3, 2^4.13, 2^4.3^2\}$. Now it is easy to see that ${\rm cod}({\rm PSL}(3,3)) \nsubseteq {\rm cod}(U_3(3))$. Also if $G/N \cong A_7$ then ${\rm cd}(A_7)=\{1, 6, 10, 14, 15,  21, 35\}$ and so ${\rm cod}(A_7)=\{1, 2^2.3.5.7, 2^2.3^2.7,\\ 2^2.3^2.5, 2^3.3.7, 2^3.3.5, 2^3.3^2\}$. Again we see that ${\rm cod}(A_7) \nsubseteq {\rm cod}(U_3(3))$. If $G/N \cong M_{11}$ then ${\rm cd}(M_{11})=\{1, 10, 11, 16, 44, 45, 55 \}$.  If $\chi(1)=16$ then ${\rm cod}(\chi(1))=3^2.5.11$ and so $3^2.5.11 \in {\rm cod}(U_3(3))$, a contradiction.  Fially if $G/N \cong J_1$ then ${\rm cd}(J_1)=\{1, 56, 76, 77, 120, 133, 209 \}$. If $\chi(1)=56$ then ${\rm cod}(\chi(1))=3.5.11.19$ and so $3.5.11.19 \in {\rm cod}(U_3(3))$, a contradiction.\\

\end{proof}

\begin{lemma}\label{b1}
Suppose that $G$ is a simple  group which is  isomorphic to alternating group $A_n$ $(n \geq 5)$ or sporadic group. Then  $4 \leq |{\rm cod}(G)|  \leq 12$ if and oly if $G$ is isomorphic to one of the following groups.  $$A_5, A_6,  A_7,  A_8,  M_{11},  M_{12},  M_{22},  M_{23},  J_1.$$

\end{lemma}

\begin{proof}
 We use the classification of finite simple groups. For the sporadic groups by \cite{CCNPW}, with the simple check we see that $G$ is isomorphic to   $M_{11}$ or $M_{12}$ or $M_{22}$ or $M_{23}$ or $J_1$ . For $A_n$ where $ 5 \leq n \leq 13$, by \cite{CCNPW} we see that $G \cong A_5$ or $A_6$ or $A_7$ or $A_8$.  Thus we may assume that $n \geq 14$. Now the characters corresponding to the non-self-associated partitions $(n-1, 1)$, $(n-2, 2)$, $(n-2, 1, 1)$, $(n-3, 3)$,  $(n-3, 2, 1)$, $(n-3, 1, 1, 1)$, $(n-4, 4)$, $(n-4, 3, 1)$, $(n-4, 2, 2 )$, $(n-4, 1, 1, 1, 1)$, $(n-5, 5)$, $(n-5, 4, 1)$, $(n-5, 3, 2)$ have distinct degrees  $n-1$, $\frac{n(n-3)}{2}$, $\frac{(n-1)(n-2)}{2}$, $\frac{n(n-1)(n-5)}{6}$, $\frac{n(n-2)(n-4)}{3}$, $\frac{(n-1)(n-2)(n-3)}{6}$, $\frac{n(n-1)(n-2)(n-7)}{24}$, $\frac{n(n-1)(n-3)(n-6)}{8}$, $\frac{n(n-1)(n-4)(n-5)}{12}$, $\frac{(n-1)(n-2)(n-3)(n-4)}{24}$, $\frac{n(n-1)(n-2)(n-3)(n-9)}{120}$, $\frac{n(n-1)(n-2)(n-4)(n-8)}{24}$, $\frac{n(n-1)(n-2)(n-5)(n-7)}{24}$. 
 
\end{proof}

\begin{theorem}\label{b2}
Suppose that $G$ is a group  and ${\rm cod}(G)={\rm cod}(U_4(2))$. If $N$ is a maximal normal subgroup of $G$ then $G/N \cong  U_4(2)$.

\end{theorem}

\begin{proof}
By our assumption $G/N$ is a  nonabelian simple group and so $|{\rm cd}(G/N)| \geq 4$. We know that  ${\rm cod}(G/N) \subseteq {\rm cod}(G)$ and ${\rm cod}(G)={\rm cod}(U_4(2))$ and so  $4 \leq |{\rm cod}(G/N)| \leq 13$.  
if $|{\rm cod}(G/N)|=13$ then $|{\rm cd}(G/N)|=|{\rm cd}(U_4(2))|$ and by Huppert's conjecture $G/N \cong U_4(2)$.  Thus we may suppose that  $4 \leq |{\rm cd}(G/N)| \leq 12$.

If $G/N$ is alternating group $A_n$ or sporadic group then by Lemma \ref{b1} $G/N$  is isomorphic to    $A_5$ or $A_6$ or $A_7$ or $A_8$ or $M_{11}$ or $M_{12}$ or $M_{22}$ or $M_{23}$ or $J_1$. By \cite{CCNPW} we see ${\rm cod}(G/N) \nsubseteq {\rm cod}(U_4(2))$, a contradiction. In what follows, we assume that $G/N$ is a simple group of Lie type.\\
  {\it Exceptional groups of Lie type}. Let $G/N$ be one of the groups $^3D_4(q)$, $G_2(q)$, $F_4(q)$, $^2B_2(q^2) (q^2=2^{2f+1})$, $E_6(q)$, $^2E_6(q)$, $E_7(q)$, $E_8(q)$, $^2G_2(q) (q=3^{2f+1})$ and $^2F_4(q) (q=2^{2f+1})$. According to \cite[Section 13.9]{C}, we see that $|{\rm cod}(G/N)| \geq 14$, where $G/N$ is isomorphic to  $F_4(q)$, $E_6(q)$, $^2E_6(q)$, $E_7(q)$, $E_8(q)$ or $^2F_4(q) $. Thus we may suppose that $G/N$ is isomorphic to  $^3D_4(q)$, $G_2(q)$, $^2B_2(q^2)$ or $^2G_2(q)$. 
  
Also   by \cite{DM} We see that $|{\rm cod}(G/N)| \geq 14$, where $G/N$ is isomorphic to $^3D_4(q)$.
If   $G/N=\hskip 0.1cm ^2B_2(2^{2f+1})$, $f \geq 1$, where $q^2=2^{2f+1}$, and $f \geq 1$, then  $|{\rm cd}(G/N)|=6$.   Now with the simple check we can see that ${\rm cod}(^2B_2(2^{2f+1}))\nsubseteq {\rm cod}(U_4(2))$, a contradiction.
 Also if  $G/N$ is isomorphic to    $G_2(q)$ then by \cite[Section 13.9]{C} we see that $G$ has cuspidal character of degree $\frac{q(q+1)^2(q^2+q+1)}{6}$, say $\phi$. Now it is easy to see that ${\rm cod}(\phi)=\frac{6q^5(q^4+q^2+1)(q-1)^2}{q^2+q+1}$. Now we see that ${\rm cod}(\phi) \notin {\rm cod}(U_4(2))$, a contradiction.  Finally if  $G$ is isomorphic to $^2G_2(q)$, then  by \cite{Wa} we see that $G/N$ has  character of degree $(q-1)(q^2-q+1)$, say $\delta$. Now it is easy to see that ${\rm cod}(\delta)=q^3(q+1)$. No we see that ${\rm cod}(\delta) \notin {\rm cod}(U_4(2))$, a contradiction.\\
{\it Type $A_n$: Linear groups.} Therefore $G/N \cong L_{n+1}(q)$, where $q$ is a prime power of $p$ and $n \geq 1$. If $n=1$ then $G/N \cong L_2(q)$. If $q=2^f$ where $f \geq 2$ then  it is easy to see that ${\rm cod(PSL(2, q))} \nsubseteq {\rm cod}(U_4(2))$. If $q=p^f$, where $p \neq 2$ and $q > 5$, then again  it is easy to see that ${\rm cod(PSL(2, q))} \nsubseteq {\rm cod}(U_4(2))$, a contradiction. Thus we may suppose $n \geq 2$,  then $G/N$ has a  character of degree $\frac{q(q^n-1)}{q-1}$, say $\chi$. Now ${\rm cod}(\chi)=q^{\frac{n^2+n-2}{2}}\Pi_{i=1}^{n}\frac{(q^{i+1}-1)(q-1)}{(n+1, q-1)(q^n-1)}$ and ${\rm cod}(\chi) \notin {\rm cod}(U_4(2))$,  a contradiction. \\
{\it Type $^2A_n$: Unitary groups.}Therefore $G/N \cong U_{n+1}(q)$, where $n \geq 2$. Now by \cite{HZ}, $G/N$ has a  character of degree $\frac{q(q^n-(-1)^n)}{q+1}$, say $\chi$. Now ${\rm cod}(\chi)=q^{\frac{n^2+n-2}{2}}\Pi_{i=1}^{n} \frac{(q^{i+1}-(-1)^{i+1})(q+1)}{(n+1, q+1)(q^n-(-1)^n)}$ and ${\rm cod}(\chi) \notin {\rm cod}(U_4(2))$,  a contradiction. \\
{\it Type $‌‌‌B_n$ or $C_n$: Odd dimensional orthogonal groups and Symplectic groups.} Therefore $G/N \cong S_{2n}(q)$ or $O_{2n+1}(q)$, where $n \geq 2$. If $n=2$ then by  \cite[Section 13.8]{C}, $G/N$ has a unipotent character of $\frac{q(q+1)^2}{2}$, say $\chi$. Now ${\rm cod}(\chi)=\frac{2q^3(q-1)^2(q^2+1)}{(2, q-1)}$ and ${\rm cod}(\chi) \notin {\rm cod}(U_4(2))$, where $q \neq 3$,  a contradiction.  If $q=3$ then $G/N \cong U_4(2)$, as wanted.  Also if $n>2$ then by \cite[Section 13.8]{C} $G/N$ has a steinberg character of degree $q^{n^2}$, say $\theta$. Now ${\rm cod}(\theta)=\frac{\Pi_{i=1}^n(q^{2i}-1)}{(2, q-1)}$ and  ${\rm cod}(\chi) \notin {\rm cod}(U_4(2))$,  a contradiction. \\
{\it Type $D_n$: Even dimensional orthogonal groups.} Therefore $G/N \cong O_{2n}^{+}(q)$, where $n \geq 2$. If $n=2$ then $G/N \cong {\rm L_2(q)} \times {\rm L_2(q)}$. Suppose that  $q=2^f$. Then ${\rm cod}(\theta)=q^2(q^2-1)$,  where $f \geq 2$ and $ \theta \in {\Irr}({\rm L_2(q)} \times {\rm L_2(q)})$. Now suppose that  $q=p^f$, where $p \neq 2$ and $p^f > 5$. Then ${\rm cod}(\theta)=\frac{q^2(q^2-1)}{4}$, where   $ \theta \in {\Irr}({\rm L_2(q)} \times {\rm L_2(q)})$.
In each case we see that ${\rm cod}(\theta) \notin {\rm cod}(U_4(2))$, a contradiction. 

 Also if $n=3$ then $G/N \cong L_4(q)$ which is isomorphic to type $A_3$. Thus we may suppose that $n \geq 4$. Now  by \cite{HZ}, $G/N$ has an irreducible character of degree  $\frac{q(q^{n-2}+1)(q^n-1)}{q^2-1}$, say $\phi$. Now ${\rm cod}(\phi)=\frac{q^{n^2-n-1}\Pi_{i=1}^{n-1}(q^{2i}-1)(q^2-1)}{(4, q^n-1)(q^{n-2}+1)}$ and ${\rm cod}(\phi) \notin {\rm cod}(U_4(2))$,  a contradiction. \\
{\it Type $^2D_n$: Even dimensional orthogonal groups.} Therefore $G/N \cong O_{2n}^{-}(q)$, where $n \geq 2$. If $n=2$ then $G/N \cong {\rm L_2(q^2)}$ which is isomorphic to type $A_1$. Also if $n=3$ then $G/N \cong U_4(q)$ which is isomorphic to type $^2A_3$. Thus we may suppose that $n \geq 4$.
Now  by \cite{HZ},  $G/N$ has an irreducible character of degree  $\frac{q(q^{n-2}-1)(q^n+1)}{q^2-1}$, say $\sigma$. Now ${\rm cod}(\sigma)=\frac{q^{n^2-n-1}\Pi_{i=1}^{n-1}(q^{2i}-1)(q^2-1)}{(4, q^n+1)(q^{n-2}-1)}$ and ${\rm cod}(\sigma) \notin {\rm cod}(U_4(2))$,  a contradiction. \\\\

\end{proof}

\begin{theorem}\label{c2}
Suppose that $G$ is a  finite group  and  ${\rm cod}(G)={\rm cod}(U_3(3))$ or ${\rm cod}(G)={\rm cod}(U_4(2))$. Then $G \cong  U_3(3)$ or $G \cong U_4(2)$.

\end{theorem}

\begin{proof}
Suppose that $N$ is the maximal normal subgroup of $G$. By Theorem~\ref{a1} and Theorem~\ref{b2}, $G/N \cong  U_3(3)$ or $G/N \cong  U_4(2)$. We show that $N=1$. Suppose to the contrary that $G$ is a counterexample with minimal order. Thus there is no groups with order less than $G$, say $H$, such that has a maximal normal subgroup, say $M$, such that $H/M \cong  U_3(3)$ or $U_4(2)$ and $M \neq 1$. Assume that  $H \unlhd G$, where $H \neq 1$ and  $H<N$. Now $G/N \cong (G/H)/(N/H)$ and so ${\rm cod (G/N)} \subseteq {\rm cod (G/H)}\subseteq {\rm cod (G)}$. Thus ${\rm cod (G/H)}={\rm cod (G)}$ and so by minimality of $G$ we get a contradiction. Thus $N$ is also a minimal normal subgroup of $G$.  If $N$ is not  abelian then  $N \cong T \times T \times \cdots T=T^n$, where $T$ is a non-abelian simple group. By Propositions \ref{f1}, \ref{f2}, there exits $1_N \neq \phi \in {\rm Irr}(N)$ which extends to some irreducible character of $G$, say $\theta$. We have ${\rm ker}(\theta)N \unlhd G$ and by the maximality of $N$ we have that either ${\rm ker}(\theta)N=N$ or ${\rm ker}(\theta)N=G$. If ${\rm ker}(\theta)N=N$ then ${\rm ker}(\theta) \subseteq N$. Since $N$ is a minimal normal subgroup we have either ${\rm ker}(\theta)=1$ or ${\rm ker}(\theta)=N$. We know that $\theta_N \neq 1$ and so we ${\rm ker}(\theta)=1$ or $G={\rm ker}(\theta) \times N$. If  ${\rm ker}(\theta)=1$ then ${\rm cod }(\theta)=|G|/\theta(1)$. Since $\theta(1) \mid |N|$, it implies that  $|N|=m\theta(1)$ for some positive integer $m$. Thus $|G|/|N|=|G|/m\theta(1)$ and so $m|G|/|N|=|G|/\theta(1)$.  Therefore $|G|/|N| \mid |G|/\theta(1)={\rm cod }(\theta)$. We know that ${\rm cod }(\theta) \in {\rm cod }(G)$ and ${\rm cod }(G)={\rm cod }(U_3(3))$ or ${\rm cod }(G)={\rm cod }(U_4(2))$. We know that $|G/N|=|U_3(3)|=2^5.3^3.7$ and $|G/N|=|U_4(2)|=2^6.3^4.5$. Now we get contradition easily.  Thus we may suppose that $G={\rm ker}(\theta) \times N$. Since $N$ is the minimal normal subgroup of $G$, it follows that ${\rm ker}(\theta)$ is the maximal normal subgroup of $G$. Now by Theorem \ref{a1} and Theorem \ref{b2} we see that $G/{\rm ker}(\theta) \cong U_3(3)$ or $G/{\rm ker}(\theta) \cong U_4(2)$. Also we have $N \cong G/{\rm ker}(\theta)$. Therefore $N \cong {\rm ker}(\theta) \cong G/{\rm ker}(\theta)$. We know that ${\rm cod }(G/{\rm ker}(\theta)) \subseteq {\rm cod }(G)$ and $ \theta \in {\rm Irr}(G/{\rm ker}(\theta))$. Now ${\rm cod}(\theta)=|G|/\theta(1)=|N|^2/\theta(1)$. Since $\theta(1) \mid |N|$ and ${\rm cod}(\theta) \in {\rm cod}(G)$ we have a contradiction.  Thus $N$ is an  abelian group. We know that  $N \leq C_G(N)$ and $C_G(N) \unlhd G$. Now since $N$ is the maximal normal subgroup of $G$ we conclude that $N=C_G(N)$ or $C_G(N)=G$. If  $C_G(N)=G$ then $N \leq Z(G)$. On the other hand if $G^{'}<G$ then by ${\rm cod}(G/G^{'}) \subseteq {\rm cod}(G)$, where $G \cong U_3(3)$ or $U_4(2)$, we see that ${\rm cod}(G)$ contains a prime, a contradiction. Thus $G$ is perfect and  $N \cap G^{'}=N$. Thus  $N \leq G^{'}$ and so $N\leq  Z(G) \cap G^{'}$.  Now since $G/N \cong U_3(3)$ or $U_4(2)$ we see thet $G$ is a central extension of $U_3(3)$ or $U_4(2)$. Also by \cite{CCNPW}, we know that ${\rm Mult}(U_3(3))=1$ and ${\rm Mult}(U_4(2)) \cong \mathbb{Z}_2$.  Thus $G \cong U_3(3)$ or $G \cong U_4(2)$ or $G \cong 2.U_4(2)$. Also by \cite{CCNPW}, we know that ${\rm cd}(2.U_4(2))=\{1, 4, 20, 36, 60, 64, 80\}$ and we have a contradiction.  Thus $G \cong U_3(3)$ or $G \cong U_4(2)$ and the assertion is hold. Therefore we may assume that $N=C_G(N)$. Now suppose that $ \theta \in {\rm Irr}(G|N)$. We show that $\theta$ is faithful. If ${\rm ker}(\theta)\neq 1$ then we know that $N \unlhd {\rm ker}(\theta)N \unlhd G$ and so by our assumption about $N$ we see that either $N ={\rm ker}(\theta)N$ or ${\rm ker}(\theta)N=G$. First suppose that ${\rm ker}(\theta)N=G$. Now ${\rm ker}(\theta) \cap N \subseteq N$ and by our assumption about $N$ we see that ${\rm ker}(\theta) \cap N =1$ or ${\rm ker}(\theta) \cap N=N$. If ${\rm ker}(\theta) \cap N=N$ then $N \subseteq {\rm ker}(\theta)  \subseteq G$, a contradiction. Thus ${\rm ker}(\theta) \cap N =1$. Thus $G={\rm ker}(\theta) \times N$. Now by considering  $G=G^{'}$ we get a contradiction. Therefore ${\rm ker}(\theta)  \leq N$ and so by the minimality of $N$ we have ${\rm ker}(\theta)=N$, a contradiction. Thus  $\theta$ is faithful for each $\theta \in {\rm Irr}(G|N)$ as asserted. Now we show that if $1_N \neq \lambda \in {\rm Irr}(N)$ then we have $|I_G(\lambda)|/\theta(1) \in {\rm cod(G)}$ for all $\theta \in {\rm Irr}(I_G (\lambda)| \lambda)$. By \cite[Theorem 6.11]{Isa}, we see that $\theta^G \in {\rm Irr}(G)$ for all $\theta \in {\rm Irr}(I_G (\lambda)| \lambda)$. If $N \leq {\rm ker}(\theta)$ then for each $n \in N$ we have $\theta(n)=\theta(1)$  and  by \cite[Theorem 6.2]{Isa} we have $\theta_N(1)=et\theta(1)$ where $t=|G:I_G(\theta)|$ and $e$ divides $|G:H|$. Now it is easy to see that for each $n \in N$ we have $\lambda(n)=\lambda(1)$, a contradiction. Thus we have $N \nleq {\rm ker}(\theta)$ and  $\theta \in {\rm Irr}(G|N)$. Now as we showed that  we see that ${\rm ker}(\theta ^G)=1$. Thus we have ${\rm cod}(\theta ^G) \in {\rm cod}(G)$, as desired. We have ${\rm cod}(G/N)={\rm cod}(G)$ and $|I_G(\lambda)|/\theta(1) \in {\rm cod}(G)$. Therefore $|I_G(\lambda)||N|/|N|\theta(1) \in {\rm cod}(G)$. By \cite[Theorem 6.15]{Isa} we have $\theta(1) \mid |I_G(\lambda)/N|$. Thus $|N|$ divides the codegree of some irreducible character of $G/N$, and so $|N| \mid |G/N|$.  

%{\bf Claim 1} If $G=U_3(3)$ then $|N| \neq 5$ and if $G=U_4(2)$ then $|N| \neq 7$.
%First suppose that  $G=U_4(2)$ and  $|N|=7$. Thus there exists $1_N \neq \lambda \in {\rm Irr}(N)$ such that $I_G(\lambda)/N$ contains a Sylow $r$-subgroup of $G/N$. Also as we showed that before we have $|I_G(\lambda)/N||N|/\theta(1) \in {\rm cod}(G)$, for all $\theta \in {\rm Irr}(I_G(\lambda)| \lambda)$. Now $|I_G(\lambda)/N|/\theta(1) \in \{2^4.3^2, 2^5, 3^3 \}$.  Now by cheking the order of maximal subgroups of $U_4(2)$ which contain $7$-subgroup Sylow subgroup in \cite{CCNPW}, we have a contradiction. Now suppose that  $G=U_3(3)$ and  $|N|=5$. Now we have $|I_G(\lambda)/N|/\theta(1) \in \{2^5.3^3, 2^3.3^3, 3^4, 2^6 \}$. 
% By cheking the order of maximal subgroups of $U_3(3)$ which contain $5$-subgroup Sylow subgroup in \cite{CCNPW} we get a contradiction. \\
%\textcolor{red}{ I think this claim is unusless. Please follow the remain proof}
We know that $C_G(N)=N$ and so $|G/N| \mid |{\rm Aut}(N)|$.   Also we know that $N$ is an elementary abelian $r$-group. If $G/N \cong U_3(3)$ then $r \in \{2, 3, 7\}$. Thus  $|{\rm Aut}(N)|$  divides  $|{\rm GL}(5, 2)|$  or $|{\rm GL}(3, 3)|$ or $|\mathbb{Z}_6|$. Now since $|N| \mid |G/N|$ we have a contradiction. Also if $G \cong U_4(2)$ then 
  $r \in \{2, 3, 5\}$.  If $|N| \mid 2^5$ or $|N| \mid 3^3$ or $|N|=5$ then  $|{\rm Aut}(N)|$  divides  $|{\rm GL}(5, 2)|$  or $|{\rm GL}(3, 3)|$ or $|\mathbb{Z}_4|$, a contradiction.  Thus we may suppose that $|N|=2^6$ or $|N|=3^4$. If $|N|=2^6$ then ${\rm Aut(N)} \cong {\rm GL(6,2)}$. Now it is easy to see that ${\rm cd}(U_4(2) \nsubseteq {\rm cd}({\rm GL}(6, 2))$, a contradiction. Also if  $|N|=3^4$ then ${\rm Aut(N)} \cong {\rm GL(4,3)}$ and again ${\rm cd}(U_4(2)) \nsubseteq {\rm cd}({\rm GL}(4, 3))$, a contradiction.
\end{proof}

\end{document}